\documentclass[12pt]{amsart}

\usepackage{amsmath}
\usepackage{amssymb}
\usepackage{esvect}
\usepackage{verbatim}
\usepackage{bm}
\usepackage{graphicx}
\usepackage{psfrag}
\usepackage{color}

\usepackage{hyperref}
\hypersetup{colorlinks=true, linkcolor=blue, citecolor=magenta, urlcolor=wine}
\usepackage{url}
\usepackage{algpseudocode}
\usepackage{fancyhdr}
\usepackage{mathtools}
\usepackage{tikz-cd}
\usepackage{xy}
\input xy
\xyoption{all}
\usepackage{stmaryrd}
\usepackage{calrsfs}

\newtheorem*{maintheorem*}{Main Theorem}
\newtheorem{theorem}{Theorem}[section]
\newtheorem*{theorem*}{Main Theorem}

\newtheorem{question}[theorem]{Question}
\newtheorem{prop}[theorem]{Proposition}
\newtheorem{conj}[theorem]{Conjecture}

\newtheorem{cor}[theorem]{Corollary}

\theoremstyle{definition}

\newtheorem{example}[theorem]{Example}
\numberwithin{equation}{section}

\newcommand{\pp}{\mathbb{P}}
\newcommand{\Q}{\mathbb{Q}}
\newcommand{\qq}{\mathbb{Q}}

\newcommand{\zz}{\mathbb{Z}}
\newcommand{\N}{\mathbb{N}}

\providecommand\ldb{\llbracket}
\providecommand\rdb{\rrbracket}

\newcommand{\gp}{\text{gp}}

\newcommand{\rr}{\mathbb{R}}
\newcommand{\nn}{\mathbb{N}}



\begin{document}
	
	\mbox{}
	\title{On the primality and elasticity of algebraic valuations of cyclic free semirings}
	
	\author{Nancy Jiang}
	\address{Milton Academy\\Milton, MA 02186}
	\email{yanan\_jiang22@milton.edu}
	
	\author{Bangzheng Li}
	\address{Christian Heritage School\\Trumbull, CT 06611}
	\email{libz2003@outlook.com}
	
	\author{Sophie Zhu}
	\address{Williamsville East High School\\ Amherst, NY 14051}
	\email{sophieiriszhu@gmail.com}

	\date{\today}

	\begin{abstract}
		A cancellative commutative monoid is atomic if every non-invertible element factors into irreducibles. Under certain mild conditions on a positive algebraic number $\alpha$, the additive monoid $M_\alpha$ of the evaluation semiring $\nn_0[\alpha]$ is atomic. The atomic structure of both the additive and the multiplicative monoids of $\nn_0[\alpha]$ has been the subject of several recent papers. Here we focus on the monoids $M_\alpha$, and we study its omega-primality and elasticity, aiming to better understand some fundamental questions about their atomic decompositions. We prove that when $\alpha$ is less than $1$, the atoms of $\nn_0[\alpha]$ are as far from being prime as they can possibly be. Then we establish some results about the elasticity of $\nn_0[\alpha]$, including that when $\alpha$ is rational, the elasticity of $M_\alpha$ is full (this was previously conjectured by S. T. Chapman, F. Gotti, and M. Gotti).
	\end{abstract}

	\maketitle

	\bigskip
	%%%%%%%%%%%
	%%%%%%%%%%%
	\section{Introduction}
	
	A cancellative and commutative (additive) monoid $M$ is said to be atomic if every non-invertible element of $M$ can be decomposed as a sum of finitely many atoms (i.e., irreducible elements). This decomposition of an element into a formal sum of atoms is called a factorization. The multiplicative flavor of the term `factorization' is due to the fact that the phenomenon of multiple decompositions into irreducibles/atoms was first studied in rings of integers of algebraic number fields (for instance, in the ring of integers $\zz[\sqrt{-5}]$, the element $6$ has two distinct atomic decompositions, namely, $6 = 2 \cdot 3 = (1 - \sqrt{-5})(1 + \sqrt{-5}$)) and later in (the multiplicative monoids of)  integral domains. 
	\smallskip
	
	Given that the primary purpose of this paper is to investigate the phenomenon of multiple factorizations in certain additive submonoids of the nonnegative cone of $\rr$, we will use additive notation. %Following standard notation from commutative algebra, we let $\nn_0[x]$ denote the semiring of polynomials in the indeterminate $x$ with coefficients in $\nn_0$. 
	For a positive real $\alpha$, we let $\nn_0[\alpha]$ denote the smallest subsemiring of $\rr$ (with identity) containing~$\alpha$, namely,
	\[
		\nn_0[\alpha] := \{ f(\alpha) \mid f(x) \in \nn_0[x] \},
	\] 
	where $\nn_0[x]$ is the semiring of polynomials in the indeterminate $x$ with coefficients in $\nn_0$. It follows immediately that $\nn_0[\alpha]$ is the additive submonoid of $\rr$ generated by the set $\{\alpha^n \mid n \in \nn_0\}$. In this paper, we will focus on the additive structure of $\nn_0[\alpha]$.
	\smallskip
	
	For rational numbers $\alpha$, the atomicity of the monoid $\nn_0[\alpha]$ was first considered in \cite[Section~5]{GG18}, and it was more thoroughly studied later in~\cite{CGG20}, where the authors investigated arithmetic and factorization invariants of $\nn_0[\alpha]$, including the set of lengths, elasticity, and omega-primality.  In the same paper, the authors proved that the set of elasticities of $\nn_0[\alpha]$ is dense in $[1, \infty)$ for every non-integer rational $\alpha > 1$, and they also proved some special cases of following conjecture.
	
	\begin{conj} \cite[Conjecture~4.8]{CGG20} \label{conj:elasticity}
		For every non-integer rational $\alpha > 1$, the set of elasticities of $\nn_0[\alpha]$ is $[1, \infty) \cap \qq$.
	\end{conj} 

	\noindent Generalizations of $\nn_0[\alpha]$ (with $\alpha$ rational) were also studied in~\cite{ABP21} from the factorization-theoretical point of view. More recently, a deeper and more systematic study of the atomic structure of $\nn_0[\alpha]$ was carried out in~\cite{CG20} for any positive real $\alpha$.
	\smallskip
	
	It is worth mentioning that atomicity and the arithmetic of factorizations in semirings of polynomials and their evaluations have earned significant attention in recent years. In \cite{CCMS09}, Cesarz et al. studied factorizations in $\rr_{\ge 0}[x]$ and $\nn_0[\alpha]$ paying special attention to the elasticity. In addition, arithmetic properties of semigroup semirings were considered by Ponomarenko in~\cite{vP15}. On the other hand, Campanini and Facchini in~\cite{CF19} studied several factorization-theoretical aspects of the multiplicative monoid of the semiring $\nn_0[x]$. More recently, Baeth and Gotti in~\cite{BG20} investigated the atomic structure of the multiplicative monoid of upper triangular matrices over positive semirings (i.e., sub-semirings of $\rr_{\ge 0}$), while a further investigation of the atomicity of both the additive and the multiplicative monoids of positive semiring was carried out by Baeth, Chapman, and Gotti~\cite{BCG21}. Even more recently, the third author studied factorizations in the additive evaluation monoids of the Laurent semiring $\nn_0[x,x^{-1}]$ at nonnegative real numbers.
	\smallskip
	
	The fundamental purpose of the present paper is to continue the study of $\nn_0[\alpha]$ initiated in~\cite{CG20}. Though, unlike~\cite{CG20}, which focused on the atomicity, here we study two arithmetic statistics of $\nn_0[\alpha]$: the omega-primality and the elasticity. In every atomic monoid, it is well known that every prime element is an atom. However, the converse does not hold in general. The omega-primality,  introduced by Geroldinger and Hassler in~\cite{GH08}, is an atomic statistic that measures how far an atom is from being a prime. In Section~\ref{sec:omega-primality}, we prove that, according to the omega-primality measurement, every atom in $\nn_0[\alpha]$ is as far from being a prime as it can possibly be (provided that $0 < \alpha < 1$). The omega-primality of the additive monoids $\nn_0[q]$ (for any positive rational $q$) was recently considered by Chapman, Gotti, and Gotti in~\cite[Section~5]{CGG20}.
	\smallskip
	
	An integral domain $R$ is called half-factorial if its multiplicative monoid is atomic and if any two factorizations of the same nonzero nonunit element have the same number of irreducibles (counting repetitions). The elasticity was introduced and first studied in the eighties by Steffan~\cite{jlS86} and Valenza~\cite{rV90} as a measurement to understand how far Dedekind domains and, specifically, rings of integers of algebraic number fields are from being half-factorial domains. Since then, the elasticity has become one of the most popular arithmetic statistics studied in factorization theory (see, for instance, \cite{AA92,AA94,AC00,CCMS09,CHM06,fHK95} and the most recent papers~\cite{BC16,GZ18,fG20,GO20,dG20}). The elasticity of a monoid can be defined similarly to that of an integral domain. The last two sections of the paper are devoted to study the elasticity of $\nn_0[\alpha]$. In Section~\ref{sec:elasticity}, we prove that the elasticity of $\nn_0[\alpha]$ is finite if and only if $\nn_0[\alpha]$ is a finitely generated monoid. In addition, we find conditions under which the set of elasticities of $\nn_0[\alpha]$ is dense in the elasticity interval. In Section~\ref{sec:rational cyclic semirings}, we focus on the elasticity of $\nn_0[\alpha]$ when $\alpha$ is rational, and we show that $\nn_0[\alpha]$ is fully elastic when $\alpha$ is a non-integer rational greater than $1$, which is Conjecture~\ref{conj:elasticity}.

	\bigskip
	%%%%%%%%%%%
	%%%%%%%%%%%
	\section{Background}
	\smallskip

	%%%%%%%%%%%%%%%
	\subsection{General Notation}
	\label{sec:background_1}
	\smallskip
	
	We let $\pp$, $\nn$, and $\nn_0$ denote the set of primes, positive integers, and nonnegative integers, respectively. If $S$ is a subset of $\rr$ and $r$ is a real number, we let $S_{\ge r}$ denote the set $\{s \in S \mid s \ge r\}$. Similarly, we use the notations $S_{> r}, S_{\le r}$, and $S_{< r}$. For a positive rational $q$, the positive integers $a$ and $b$ with $q = a/b$ and $\gcd(a,b) = 1$ are denoted by $\mathsf{n}(q)$ and $\mathsf{d}(q)$, respectively. 
	\smallskip
	
	Let $f(X)\in\mathbb{Q}[X]$ be a nonzero polynomial. The \emph{support} of $f(X)$, denoted by $\text{supp} \, f(X)$, is the set of exponents of the monomials appearing in the canonical representation of $f(X)$; that is,
	\[
		\text{supp} \, f(X) := \{n \in \nn_0 : f^{(n)}(0) \neq 0 \},
	\]
	where $f^{(n)}(X)$ denotes the $n$-th formal derivative of $f(X)$. The \emph{order} of $f(X)$ is defined to be the minimum of $ \text{supp} \, f(X)$, which is the largest $d \in \nn_0$ such that $f(X) = X^d g(X)$ for some $g(X) \in \qq[X]$. Now suppose that $f(X)$ is monic. Let $\ell$ be the smallest positive integer such that $\ell \cdot f(X) \in \mathbb{Z}[X]$. Then there exist unique $p(X), q(X) \in \mathbb{N}_0[X]$ such that $\ell f(X) = p(X) - q(X)$ and that $p(X)$ and $q(X)$ share no monomials of the same degree (that is, the greatest common divisor of $p(X)$ and $q(X)$ in the free commutative monoid $(\N_0[X],+)$ is $0$). We call the pair $(p(X), q(X))$ the \emph{minimal pair} of $f(X)$. In addition, if $\alpha$ is a real algebraic number, the \emph{minimal pair of $\alpha$} is defined to be the minimal pair of its minimal polynomial over $\Q$.
	\smallskip
	
	According to Descartes’ Rule of Signs, the number of variations of sign of a polynomial $f(X) \in \rr[X]$ has the same parity as and is at least the number of positive roots of $f(X)$, counting multiplicity. Also, it was proved by D. R. Curtiss~\cite{dC18} that there exists a polynomial $\mu(X) \in \rr[X]$, which he called a Cartesian multiplier, such that the number of variations of sign of the product polynomial $\mu(X)f(X)$ equals the number of positive roots of $f(X)$, counting multiplicity.

	\medskip
	%%%%%%%%%%%%%%%%%%%%%%%%%%%%%
	\subsection{Commutative Monoids and Factorizations}
	
	In the scope of this paper, a \emph{monoid} is a cancellative and commutative semigroup with an identity element. In addition, monoids here will be written additively unless we explicitly state otherwise. In particular, the identity element of a monoid is $0$. Let $M$ be a monoid. We set $M^\bullet = M \setminus \{0\}$. It is clear that $0$ is an invertible element of $M$. The monoid $M$ is called \emph{reduced} if $0$ is its only invertible element. For every $\alpha \in \rr_{> 0}$, the additive submonoid $\nn_0[\alpha]$ of $\rr_{\ge 0}$, which plays a fundamental role in this paper, is reduced. In general every additive submonoid of $\rr_{\ge 0}$ is reduced. Additive submonoids of $\qq_{\ge 0}$, also known as \emph{Puiseux monoids}, are also reduced; they will also be of importance in this paper. From this point on, we tacitly assume that all monoids here (that are not groups) are reduced.
	\smallskip
	
	For a subset $S$ of $M$, we let $\langle S \rangle$ denote the submonoid of $M$ \emph{generated} by $S$, that is, $\langle S \rangle$ is the intersection of all submonoids of $M$ containing $S$. We say that a monoid is \emph{finitely generated} if it can be generated by a finite set. The \emph{difference group} $\gp(M)$ of $M$ is the unique abelian group (up to isomorphism) such that any abelian group containing a homomorphic image of $M$ will also contain a homomorphic image of $\gp(M)$. If $M$ is an additive submonoid of $\rr_{\ge 0}$, then
	\[
		\gp(M) = \{r-s \mid r, s \in M\}.
	\]
	
	A nonzero element $a \in M$ is called an \emph{atom} if whenever $a = b+c$ for some $b,c \in M$ either $b = 0$ or $c = 0$. Following common notation, we let $\mathcal{A}(M)$ denote the set consisting of all atoms of $M$. Following \cite{pC68}, we say that $M$ is \emph{atomic} if every non-invertible element of $M$ can be expressed as a sum of atoms. For $b,c \in M$, we say that $b$ \emph{divides} $c$ in $M$ if there exists $b' \in M$ such that $c = b + b'$, in which case we write $b \mid_M c$. An element $p \in M^\bullet$ is called a \emph{prime} provided that for all $b,c \in M$, the fact that $p \mid_M b + c$ implies that either $p \mid_M b$ or $p \mid_M c$. It is not hard to verify that every prime is an atom. The converse does not hold in general; for instance, one can readily check that $N := \{0\} \cup \zz_{\ge 2}$ is an atomic monoid with $\mathcal{A}(N) = \{2,3\}$ even though $N$ does not contain any primes. A monoid is called an \emph{atom-prime} monoid or an \emph{AP-monoid} if every atom is prime.
	\smallskip
	
	In an atomic monoid, an element can be decomposed as a sum of atoms in multiple ways, and to understand these multiple atomic decompositions, the notion of a factorization plays a central role. Assume for the rest of this section that $M$ is an atomic monoid. We let $\mathsf{Z}(M)$ denote the free (commutative) monoid on $\mathcal{A}(M)$, that is, $\mathsf{Z}(M)$ is the monoid consisting of all formal sums of atoms of $M$. Now for each $x \in M$, we let $\mathsf{Z}(x)$ denote the set of all formal sums $z := a_1 +  \cdots + a_\ell$ in the free monoid $\mathsf{Z}(M)$ with $a_1, \dots, a_\ell \in \mathcal{A}(M)$ satisfying $a_1 + \dots + a_\ell = x$ in $M$. We then call $z$ a \emph{factorization} of $x$ and~$\ell$ the \emph{length} of $z$. The monoid $M$ is called a \emph{unique factorization monoid} (\emph{UFM}) if every nonzero element of $M$ has a unique factorization, that is, $|\mathsf{Z}(x)| = 1$ for all $x \in M^\bullet$. It is well known that a monoid is a UFM if and only if it is an atomic AP-monoid. Following Anderson, Anderson, and Zafrullah~\cite{AAZ90} and Halter-Koch~\cite{fHK92}, we say that $M$ is a \emph{finite factorization monoid} (\emph{FFM}) if $\mathsf{Z}(x)$ is a finite set for all $x \in M^\bullet$. Every UFM is clearly an FFM, and it follows from \cite[Proposition~2.7.8]{GH06} that every finitely generated monoid is an FFM. We denote the length of $z$ by~$|z|$, and then, for each $x \in M$, we set
	\[
		\mathsf{L}(x) := \{ |z| \mid z \in \mathsf{Z}(x) \}.
	\]
	The set $\mathsf{L}(x)$ will play an important role in the coming sections. Following Zaks~\cite{aZ76}, we say that $M$ is a \emph{half-factorial monoid} (\emph{HFM}) if any two factorizations of the the same nonzero element of $M$ have the same length, that is, $|\mathsf{L}(x)| = 1$ for all $x \in M^\bullet$. Note that every UFM is an HFM. Following~\cite{AAZ90} and~\cite{fHK92}, we say that $M$ is a \emph{bounded factorization monoid} (\emph{BFM}) if $\mathsf{L}(x)$ is a finite set for all $x \in M^\bullet$. Observe that every FFM is a BFM, and every HFM is a BFM. The bounded and finite factorization properties were recently surveyed in~\cite{AG21} in the context of integral domains. A brief introduction to factorization theory in commutative monoids can be found in~\cite{GZ20}, while an extensive background material on factorization theory in both commutative monoids and integral domains can be found in~\cite{GH06}.

\bigskip
%%%%%%%%%
%%%%%%%%%
\section{Primality}
\label{sec:omega-primality}

Let $M$ be an atomic monoid that is not a group. The \emph{omega-primality} of $a \in \mathcal{A}(M)$, denoted by $\omega(a)$, is the smallest $n \in \nn \cup \{\infty\}$ such that the following condition holds: if $a \mid_M \sum_{i=1}^k a_i$ for some $a_1, \dots, a_k \in \mathcal{A}(M)$, then there exists $S \subseteq \ldb 1, k \rdb$ with $|S| \le n$ such that $a \mid_M \sum_{i \in S} a_i$. Observe that an atom $a$ of $M$ is prime if and only if $\omega(a) = 1$. Therefore we can think of the omega-primality of an atom as a measure of how far it is from being a prime. The \emph{omega-primality} of $M$ is then defined as
\[
	\omega(M) := \sup \{\omega(a) \mid a \in \mathcal{A}(M)\}.
\]
 Therefore $\omega(M) = 1$ if and only if every atom of $M$ is prime, and so the omega-primality of $M$ measures how $M$ is from being an AP-monoid or a UFM. We say that~$M$ is an \emph{anti-prime} monoid provided that $\omega(a) = \infty$ for every $a \in \mathcal{A}(M)$; that is, atoms in~$M$ are as far from being prime as they can possibly be. In particular, an anti-prime monoid has infinite omega-primality.

It was proved in \cite[Thereom~5.6]{CGG20} that for any rational $q$ that is not an integer, if $\nn_0[q]$ is atomic, then its omega-primality is infinite. As the following theorem indicates, when $\nn_0[\alpha]$ is atomic and $\alpha$ is an algebraic number in the interval $(0,1)$, the monoid $\nn_0[\alpha]$ is an anti-prime monoid and, in particular, has infinite omega-primality.

\begin{theorem} \label{thm:omega primality}
	Let $\alpha \in \rr_{> 0}$ with $0 < \alpha < 1$ be an algebraic number such that the monoid $\nn_0[\alpha]$ is atomic. Then $\nn_0[\alpha]$ is an anti-prime monoid.
\end{theorem}

\begin{proof}
	Let $p(X)$ and $q(X)$ be polynomials in $\nn_0[X]$ with $p(\alpha) = q(\alpha)$ such that $\text{ord} \, p(X) > 0$ and $\text{ord} \, q(X) = 0$; for instance, we can take $p(X)$ and $q(X)$ to be the polynomials in the minimal pair of $\alpha$. Since $\alpha$ is a positive algebraic number less than~$1$, it follows from \cite[Theorem~4.1]{CG20} that $\mathcal{A}(\nn_0[\alpha]) = \{\alpha^n \mid n \in \nn_0\}$. Set $M := \nn_0[\alpha]$ and, for each $n \in \nn$, we set
	\[
		M_n := \langle \alpha^j \mid j \in \nn_{\ge n} \rangle.
	\]
	Fix $k \in \nn_0$. We will argue by induction that, for each $n \in \nn_0$, there exists $\beta_n \in M_{k+n}$ such that $\alpha^k \mid_M \beta_n$. For $n=0$, we can take $\beta_0 = \alpha^k$. Now suppose that for $n \in \nn_0$, there exists $\beta_n \in M_{k+n}$ with $\alpha^k \mid_M \beta_n$. Take $m \in \nn_0$ such that $\beta_n = m \alpha^{k+n} + \beta'_{n+1}$ for some $\beta'_{n+1} \in M_{k+ (n+1)}$. Set $\beta_{n+1} := m \alpha^{k+n} q(\alpha) + \beta'_{n+1}$. Since $m \alpha^{k+n} q(\alpha) = m \alpha^{k+n} p(\alpha) \in M_{k + (n+1)}$, it follows that $\beta_{n+1} \in M_{k + (n+1)}$. In addition,
	\[
		\beta_{n+1} = m \alpha^{k+n}(q(\alpha) - 1) + (m \alpha^{k+n} + \beta'_{n+1}) = m \alpha^{k+n}(q(\alpha) - 1) + \beta_n, 
	\]
	and so $\beta_n \mid_M \beta_{n+1}$. Because $\alpha^k \mid_M \beta_n$, we obtain that $\alpha^k \mid_M \beta_{n+1}$, as desired.
	
	Let us show now that $\omega(\alpha^k) = \infty$ for every $k \in \nn_0$. To do this, suppose, by way of contradiction, that $\omega(\alpha^k) = K \in \nn$ for some $k \in \nn_0$. Take $N \in \nn$ large enough so that $K \alpha^N < \alpha^k$. By the argument given in the previous paragraph, we can take $x \in M_N$ such that $\alpha^k \mid_M x$. Since $x \in M_N$, we can write $x = a_1 + \dots + a_\ell$, where $a_1, \dots, a_\ell \in \{\alpha^j \mid j \in \nn_{\ge N}\}$. It follows from $\alpha^k \mid_M x$ and $\omega(\alpha^k) = K$ that there exists $S \subseteq \ldb 1, \ell \rdb$ with $|S| \le K$ such that $\alpha^k \mid_M \sum_{i \in S} a_i$. However, this would imply that $\alpha^k \le \sum_{i \in S} a_i \le |S| \alpha^N \le K \alpha^N < \alpha^k$, a contradiction. As a consequence, $\omega(\alpha^k) = \infty$ for every $k \in \nn_0$, from which we can conclude that $\nn_0[\alpha]$ is an anti-prime monoid.
\end{proof}

\begin{cor} \label{cor:infinite omega-primality}
	Let $\alpha \in \rr_{> 0}$ with $0 < \alpha < 1$ be an algebraic number such that $\nn_0[\alpha]$ is atomic. Then $\omega(\nn_0[\alpha]) = \infty$.
\end{cor}

For an algebraic number $\alpha$ in the interval $(0,1)$ such that $(\nn_0[\alpha],+)$ is atomic, we have proved in Theorem~\ref{thm:omega primality} that every atom has infinite omega-primality. Even in the context of additive submonoids of $\qq_{\ge 0}$, we can find examples of non-finitely generated monoids whose omega-primality is finite. Our next example illustrates this observation. The \emph{conductor} $c(M)$ of a Puiseux monoid $M$ is defined by
\[
	c(M) := \inf \{ r \in \rr_{\ge 0} \cup \{\infty\} \mid M_{\ge r} = \text{gp}(M)_{\ge r} \}.
\]

\begin{example}
	Fix $q \in(1,2) \cap \qq$, and let $M_q$ denote the Puiseux monoid $\langle[1,q] \cap \qq \rangle$. We can readily check that $M_q$ is atomic with $\mathcal{A}(M_q) = [1,q] \cap \mathbb{Q}$. On the other hand, it is not hard to verify that the conductor of $M_q$ is finite. Moreover, if $c := c(M_q)$, then one can further argue that $c\in \nn$ and
	\[
		M_q = \{0\} \bigcup \bigg( \bigcup_{k=1}^{c-1} [k,kq] \cap \qq \bigg) \bigcup \, \big( [c,\infty) \cap \qq \big),
	\]
	where $kq < k+1$ for every $k \in \ldb 1, c-1 \rdb$ and $cq \ge c+1$. Observe that these inequalities guarantee that $c = \lceil \frac{1}{q-1} \rceil$.
	\smallskip
	
	We claim that $\omega(M)$ is finite. Proving this claim amounts to showing that the identity $\omega(a) = c+\lceil a \rceil$ holds for all $a \in \mathcal{A}(M_q)$. Fix $a \in \mathcal{A}(M_q)$ and set $n := \lceil a \rceil$ (note that $n = 1$ if $a = 1$, and $n = 2$ if $a > 1$). Then for any nonzero elements $q_1, \ldots, q_{c + n} \in M_q$, we have $q_1 + \dots + q_{c + n} \geq c + n$. This implies that $a \mid_{M_q} q_1 + \dots + q_{c + n}$ because $(q_1 + \dots + q_{c + n}) - a \geq c + (n - a) \ge c$. It follows that $\omega(a) \leq c + n$. To argue the reverse inequality, first observe that $c < 1 + \frac{1}{q-1}$, which implies that $(c-1)q + a < c+a$. On the other hand, the fact that $n q > a$ implies $\frac{(c-1)q + a}{c+ n -1} < q$, and it is clear that $\frac{c+a}{c + n - 1} > 1$. Then there exists an element $b \in M_q$ such that 
	\[
		\frac{(c-1)q + a}{c + n - 1} < b <\frac{c + a}{c + n - 1}.
	\]
	Consequently, $(c-1)q < (c + n - 1)b - a < c$, so $(c + n - 1)b - a \notin M_q$. As a consequence, $a \nmid_{M_q} (c + n - 1)b$. We then conclude from $a \mid_{M_q} (c + n)b$ that $\omega(a) = c + n$.
\end{example}
	
Based on Corollary~\ref{cor:infinite omega-primality} and \cite[Theorem~5.6]{CGG20}, we conclude this section with the following conjecture.

\begin{conj}
	Let $\alpha$ be a positive algebraic number such that $\nn_0[\alpha]$ is atomic. Then $\omega(\nn_0[\alpha]) < \infty$ if and only if $\alpha \in \nn$.
\end{conj}

\bigskip
%%%%%%%%%%
%%%%%%%%%%
\section{Elasticity}
\label{sec:elasticity}

In this section we study the elasticity of the monoids $\nn_0[\alpha]$ for positive algebraic numbers $\alpha$. First, let us formally define the notion of elasticity. Let $M$ be an atomic monoid. The \emph{elasticity} of $x \in M^\bullet$, denoted by $\rho(x)$, is defined as
\[
	\rho(x) := \frac{\sup \mathsf{L}(x)}{\min \mathsf{L}(x)}.
\]
In addition, we set $\rho(M) := \sup \{\rho(x) \mid x \in M^\bullet\}$ and call it the \emph{elasticity} of $M$. Notice that $\rho(M) \in \rr_{\ge 1} \cup \{\infty \}$. Furthermore, observe that $\rho(M) = 1$ if and only if $M$ is an HFM. As a result, the elasticity provides a measurement for how far an atomic monoid is from being an HFM. The elasticity of $M$ is called \emph{accepted} if there is an element $x \in M^\bullet$ such that $\rho(M) = \rho(x)$.

\begin{theorem} \cite[Theorem~3.1.4]{GH06} \label{thm:reduced fg monoids have finite elasticity}
	Every reduced finitely generated monoid has accepted elasticity. %and, therefore, finite elasticity.
\end{theorem}

It was proved in~\cite[Theorem~5.4]{CG20} that for any positive algebraic number $\alpha$, the equality $\rho(\nn_0[\alpha]) = 1$ holds  if and only if $\nn_0[\alpha]$ is a UFM, which happens precisely when the degree of the minimal polynomial of $\alpha$ equals $|\mathcal{A}(\nn_0[\alpha])|$. This seems to be the only fact known so far about the elasticity of the monoids $\nn_0[\alpha]$. In this section, we offer some insights into the elasticity of the monoids $\nn_0[\alpha]$.
\smallskip

Let us provide, for any positive algebraic number $\alpha$, a necessary condition for the elasticity of $\nn_0[\alpha]$ to be finite.

\begin{prop} \label{prop:elasticity}
	Let $\alpha \in \rr_{> 0}$ be an algebraic number such that $\nn_0[\alpha]$ is atomic. Then $\rho(\nn_0[\alpha]) < \infty$ if and only if $\nn_0[\alpha]$ is finitely generated.
\end{prop}

\begin{proof}
	%If an atomic monoid $M$ is not factorial, then $\max \{2, \rho(M)\} \le t(M)$ (see \cite[Theorem~1.6.6]{GH06}). Therefore every globally tame monoid has finite elasticity.
	It follows from \cite[Proposition~2.7.8]{GH06} that every finitely generated monoid is an FFM and so a BFM. Therefore, if $\nn_0[\alpha]$ is finitely generated, then $\rho(\beta) < \infty$ for all $\beta \in \nn_0[\alpha]^\bullet$. In particular, by Theorem~\ref{thm:reduced fg monoids have finite elasticity}, every reduced finitely generated monoid has finite elasticity. Thus, the reverse implication follows.
	
	For the direct implication, assume that $\nn_0[\alpha]$ is not finitely generated. Then it follows from \cite[Theorem~4.1]{CG20} that $\mathcal{A}(\nn_0[\alpha]) = \{\alpha^n \mid n \in \nn_0\}$. Since $\nn_0[\alpha]$ is not finitely generated, $\alpha \notin \nn$. We will construct a sequence $(\beta_n)_{n \in \N}$ with $\beta_n \in \nn_0[\alpha]$ for every $n \in \nn$ such that $\sup \{\rho(\beta_n) \mid n \in \nn)\} = \infty$. Let $(p(X), q(X))$ be the minimal pair of $\alpha$. Then $z_1 := p(\alpha)$ and $z_2 := q(\alpha)$ are two distinct factorizations of the same nonzero element $\beta_1 \in \nn_0[\alpha]$ and have lengths $p(1)$ and $q(1)$, respectively. Since $1$ is not a root of the minimal polynomial of $\alpha$, we see that $p(1) \neq q(1)$, which means $z_1$ and $z_2$ have different lengths. Without loss of generality, suppose that $|z_1| < |z_2|$. For each $n \in \nn$, set $\beta_n = \beta_1^n.$ Then, for every $n \in \nn$, both $p(\alpha)^n$ and $q(\alpha)^n$ yield factorizations of $\beta_n$ whose lengths are $p(1)^n$ and $q(1)^n$, respectively. Therefore
	\begin{equation} \label{eq:elasticity bound}
		\rho(\nn_0[\alpha]) \ge \rho(\beta_n) = \frac{\sup \mathsf{L}(\beta_n)}{\min \mathsf{L}(\beta_n)} \ge \frac{q(1)^n}{p(1)^n} =  \bigg( \frac{q(1)}{p(1)} \bigg)^n
	\end{equation}
	for every $n \in \nn$. Since $q(1)/p(1) > 1$, taking the limits of both sides of~\eqref{eq:elasticity bound}, we see that $\rho(\nn_0[\alpha]) = \infty$, which concludes the proof.
\end{proof}
	
\begin{cor}
	Let $\alpha \in \rr_{> 0}$ be an algebraic number such that $\alpha < 1$. If $\nn_0[\alpha]$ is atomic, then $\rho(\nn_0[\alpha]) = \infty$.
\end{cor}

\begin{proof}
	Suppose that $\nn_0[\alpha]$ is atomic. If $\alpha < 1$, then $0$ is a limit point of $\nn_0[\alpha]^\bullet$ and, therefore, $\nn_0[\alpha]$ cannot be finitely generated. Hence it follows from Proposition~\ref{prop:elasticity} that $\rho(\nn_0[\alpha]) = \infty$.
\end{proof}

In general, it seems significantly difficult to establish a general formula for the elasticity of $\nn_0[\alpha]$ when this monoid is finitely generated. However, when $|\mathcal{A}(\nn_0[\alpha])| = \deg m_{\alpha}(X)+1$, we are able to provide a formula for the elasticity.

\begin{prop}
	Let $\alpha$ be a positive algebraic number with minimal polynomial given by $m_{\alpha}(X)\coloneqq X^n + \sum_{i=0}^{n-1} a_i X^i$, where $a_0, \dots, a_{n-1} \in \zz$. Let $\left(p(X),q(X)\right)$ be the minimal pair of $m_\alpha(X)$. If $|\mathcal{A}(\nn_0[\alpha])| = \deg m_{\alpha}(X)+1$, then
	\[
		\rho(\nn_0[\alpha])\ =\ \max\left\{\frac{p(1)}{q(1)},\frac{q(1)}{p(1)}\right\}.
	\]
\end{prop}

\begin{proof}

Since $|\mathcal{A}(\nn_0[\alpha])| = \deg m_{\alpha}(X)+1$, it follows from \cite[Theorem 4.1]{CG20} that $\nn_0[\alpha]$ is atomic with $\mathcal{A}(\nn_0[\alpha]) = \{ 1, \alpha, \alpha^2, \ldots, \alpha^n\}$. Thus, $\nn_0[\alpha]$ is finitely generated, and so \cite[Proposition~2.7.8]{GH06} guarantees that $\nn_0[\alpha]$ is an FFM. On the other hand, from the condition $|\mathcal{A}(\nn_0[\alpha])| = \deg m_{\alpha}(X)+1$, we see that $\alpha \ne 1$, which implies $p(1)\ne q(1)$. Without loss of generality, we can assume $p(1) < q(1)$. Since $\nn_0[\alpha]$ is finitely generated, Theorem~\ref{thm:reduced fg monoids have finite elasticity} allows us to choose two nonzero polynomials $f,g \in \nn_0[X]$ with $\max \{\deg f(X), \deg g(X)\} \le n$ and $f(\alpha) = g(\alpha)$ such that
\[
	 \rho(\nn_0[\alpha]) =\max \left\{  \frac{f(1)}{g(1)}, \frac{g(1)}{f(1)} \right\} < \infty,
\]
where the last inequality follows from the fact that $\nn_0[\alpha]$ is an FFM. Suppose, for a contradiction, that $f$ and $g$ share a common term, namely $X^k$. Then the polynomials $f_0(X) := f(X) - X^k$ and $g_0(X) := g(X) - X^k$ are both in $\nn_0[X]$, and so $f_0(\alpha)$ and $g_0(\alpha)$, seen as formal sums, give two factorizations of the same element of $\nn_0[\alpha]$. Therefore if $f(1) > g(1)$, then
\[
	\rho(\nn_0[\alpha]) \ge \frac{f_0(1)}{g_0(1)} = \frac{f(1) - 1}{g(1) - 1} > \frac{f(1)}{g(1)},
\]
which contradicts the selection of $f$ and $g$. In a similar way, we can obtain a contradiction when $g(1) > f(1)$. Thus, we conclude that $f$ and $g$ do not share any common term.

Now we have $m_\alpha(X) = p(X) - q(X)$ and $f(X) - g(X) = c m_\alpha(X)$ for some $c\in \mathbb{Z}$. Interchanging the roles of $f$ and $g$ if necessary, we can assume that $c \in \nn_0$. The equality $f(X) + c q(X) = g(X) + c p(X)$, along with the fact that $f$ and $g$ do not share any common term, guarantees that $f(X) = c p(X)$ and $g(X) = c q(X)$. Observe that $c > 0$ because $f$ and $g$ are nonzero. Finally,
\[
	\rho(\nn_0[\alpha]) = \max \left\{ \frac{f(1)}{g(1)}, \frac{g(1)}{f(1)} \right\} = \frac{c q(1)}{c p(1)} = \frac{q(1)}{p(1)}.
\]
\end{proof}

\medskip
%%%%%%%%%%%%%%%%%%%
\subsection{The Set of Elasticities}

The set $R(M) = \{ \rho(x) \mid x \in M^\bullet \}$ is called the \emph{set of elasticities} of $M$. It turns out that if $\alpha > 1$ is an algebraic number and is not an algebraic integer, then $R(\nn_0[\alpha])$ is dense in $\rr_{\ge 1}$.
 
 \begin{prop} \label{prop:density of the set of elasticities}
 	Let $\alpha$ be an algebraic number such that $\alpha > 1$. If $\alpha$ is not an algebraic integer, then $\nn_0[\alpha]$ is densely elastic.
 \end{prop}
 
 \begin{proof}
 	Since $\alpha > 1$, we see that $0$ is not a limit point of $\nn_0[\alpha]$, and so it follows from \cite[Proposition~4.5]{fG19} that $\nn_0[\alpha]$ is a BFM. In addition, it follows from \cite[Theorem~4.1]{CG20} that $\mathcal{A}(\nn_0[\alpha]) = \{\alpha^n \mid n \in \nn_0\}$. Take $\beta \in \nn_0[\alpha]$ to be a nonzero element, and set
 	\[
 		\ell := \min \mathsf{L(\beta)} \quad \text{ and } \quad L := \max \mathsf{L(\beta)}.
 	\]
 	Let $f(X)$ be a polynomial in $\nn_0[X]$ such that $f(\alpha) = \beta$, and take $n \in \nn$ such that $n > \deg f(X)$ and $\alpha^{n+1} > \alpha^n + \beta$. Now set $\beta' := \beta + \alpha^n \in \nn_0[\alpha]$. 
 	
 	We claim that the atom $\alpha^n$ appears in every factorization of $\beta'$. Suppose, by way of contradiction, that this is not the case. Then there exists $g(X) \in \nn_0[X]$ such that $n \notin \text{supp} \, g(X)$ such that $f(\alpha) + \alpha^n = g(\alpha)$. Observe that $\deg g(X) < n$ as otherwise $g(\alpha) \ge \alpha^{\deg g(X)} \ge \alpha^{n+1} > f(\alpha) + \alpha^n$. Since $n > \max \{\deg f(X), \deg g(X)\}$, the polynomial $X^n + f(X) - g(X) \in \zz[X]$ is a monic polynomial of degree $n$ having~$\alpha$ as a root. However, this contradicts that $\alpha$ is not an algebraic integer.
 	
 	Since $\alpha^n$ appears in every factorization of $\beta'$, it follows that $\mathsf{Z}(\beta') = \alpha^n + \mathsf{Z}(\beta)$ and, consequently,
 	\[
 		\rho(\beta') = \frac{\max \mathsf{L}(\beta')}{\min \mathsf{L}(\beta')} = \frac{1 + \max \mathsf{L}(\beta)}{1 + \min \mathsf{L}(\beta)} = \frac{L+1}{\ell+1}. 
 	\]
 	By induction, we obtain that $\frac{M+n}{m+n} \in R(\nn_0[\alpha])$ for every $n \in \nn$. Finally, suppose that $(\beta_m)_{m \in \nn}$ is a sequence in $\nn_0[\alpha]$ such that $\sup \{\rho(\beta_m) \mid m \in \nn \} = \rho(\nn_0[\alpha])$, and set $\ell_m := \min \mathsf{L}(\beta_m)$ and $L_m := \max \mathsf{L}(\beta_m)$. By virtue of \cite[Lemma~5.6]{GO20},
 	\[
 		S := \bigg\{ \frac{L_m + n}{\ell_m + n} \ \bigg{|} \ m,n \in \nn \bigg\}
 	\]
 	is dense in $[1, \rho(\nn_0[\alpha])]$ (resp., in $\rr_{\ge 1}$) if $\rho(\nn_0[\alpha]) < \infty$ (resp., $\rho(\nn_0[\alpha]) = \infty$). Since $S$ is a subset of $R(\nn_0[\alpha])$, the set $R(\nn_0[\alpha])$ is dense in $\rr_{\ge 1}$ and, therefore, $\nn_0[\alpha]$ is densely elastic.
 \end{proof}
 
 Based on Proposition \ref{prop:density of the set of elasticities}, we wonder whether the set of elasticities of $\nn_0[\alpha]$ is dense in $\rr_{\ge 1}$ provided that $\nn_0[\alpha]$ has infinite elasticity.
 
 %The monoid~$M$ is called \emph{fully elastic} if $R(M)$ is an interval in the poset $\qq \cup \{\infty \}$. %We say that $M$ has \emph{accepted elasticity} if $\rho(M) = \rho(x)$ for some $x \in M^\bullet$.

\begin{question}
	Is $R(\nn_0[\alpha])$ dense in $\rr_{\ge 1}$ provided that $\rho(\nn_0[\alpha]) = \infty$?
\end{question}

\bigskip
%%%%%%%%%%%%%%%%%%%
\section{Rational Cyclic Semirings}
\label{sec:rational cyclic semirings}

The monoid~$M$ is called \emph{fully elastic} if $R(M)$ is an interval in the poset $\qq \cup \{\infty \}$. We proceed to prove that if  $q \in \qq_{> 1}$, then the monoid $\nn_0[q]$ is fully elastic. This answers a conjecture posed by Chapman et al. \cite[Conjecture~4.8]{CGG20} (see \cite[Conjecture~4.11]{fG22} for another recent conjecture related to the atomicity of $\nn_0[q]$). 

\begin{theorem}
	For any $q \in \qq_{> 1}$, the monoid $\nn_0[q]$ is fully elastic.
\end{theorem}

\begin{proof}
	Fix $q \in \qq_{> 1}$. If $q \in \nn$, then $\nn_0[q] = \nn_0$ is a UFM, and so it is trivially fully elastic. Then we assume that $q \notin \nn$. Write $q = a/b$ for some $a,b \in \nn$ with $\gcd(a,b) = 1$. If $a = b+1$, then $\nn_0[q]$ is fully elastic by \cite[Proposition~4.4]{CGG20}. Thus, we assume that $a-b > 1$.
	
	To argue that $\nn_0[q]$ is fully elastic, fix a rational number $s/t$, where $s,t \in \nn$ such that $s > t \ge 1$ and $\gcd(s,t) = 1$.
	\bigskip
	
	\noindent \textit{Claim:} If there exists $\beta \in \nn_0[q]$ such that $\ell := \min \mathsf{L}(\beta)$ and $L := \max \mathsf{L}(\beta)$ satisfy that $s-t \mid_\nn tL - s\ell$, then there exists $\beta' \in \nn_0[q]$ such that $\rho(\beta') = s/t$.
	\smallskip
	
	\noindent \textit{Proof of Claim:} Fix $\beta \in \nn_0[q]$, and set $\ell := \min \mathsf{L}(\beta)$ and $L := \max \mathsf{L}(\beta)$. Now suppose that $s-t \mid_\nn tL - s\ell$. Since $b \ge 2$, the rational $q$ is not an algebraic integer, and so we can proceed as we did in the proof of Proposition~\ref{prop:density of the set of elasticities} to verify that
	\begin{equation} \label{eq:set of elasticities}
		S := \bigg\{ \frac{L + c}{\ell + c} \ \bigg{|} \ c \in \nn \bigg\} \subseteq R(\nn_0[q]).
	\end{equation}
	Since $s-t \mid_\nn tL - s\ell$, we can take $c \in \nn$ such that $tL - s\ell = c(s-t)$ and, from this equality, we obtain that
	\[
		\frac{s}{t} = \frac{L + c}{ \ell + c} \in S.
	\]
	Therefore, by virtue of~\eqref{eq:set of elasticities}, there exists $\beta' \in \nn_0[q]$ such that $\rho(\beta') = s/t$, from which the claim follows.
	\smallskip
	
	Now for each $k \in \nn$, set $\ell_k := \min \mathsf{L}(a^k)$ and $L_k = \max \mathsf{L}(a^k)$. Since $q^k \in \mathcal{A}(\nn_0[q])$, we see that $\ell_k = \min \mathsf{L}(b^k q^k) \le b^k$. Similarly, $1 \in \mathcal{A}(\nn_0[q])$ ensures that $L_k \ge a^k$. As a result, if we take $N \in \nn$ such that $q^N > s/t$, then we obtain that $t L_k - s \ell_k \ge ta^k - sb^k \ge 1$ for every $k \ge N$. Now let $r_k$ denote the integer in $\ldb 0, s-t-1 \rdb$ such that
	\[
		r_k \, \equiv \, t L_k - s \ell_k \! \pmod{s-t}.
	\]

	Now fix $k \in \nn$. Since $1 = \min \mathcal{A}(\nn_0[q])$, we can easily see that $a^k \cdot 1$ is the only maximum-length factorization of $a^k$ in $\nn_0[q]$. Suppose now that
	\[
		z_k := \sum_{i=j}^{m_k} n_i q^i \in \mathsf{Z}(a^k)
	\]
	is a minimum-length factorization of $a^k$, where $j \in \nn_0$ and $n_j \neq 0$. Observe that $j \ge k$ as, otherwise, the equality
	\[
		a^{k-j} b^{m_k} = n_j b^{m_k - j} + \sum_{i = j+1}^{m_k} n_i a^{i-j} b^{m_k - i}
	\]
	would imply that $a \mid_\nn n_j$ and, after replacing $a q^j$ by $b q^{j+1}$, we would obtain another factorization $z'$ in $\mathsf{Z}(a^k)$ with $|z'| < |z|$. This observation, along with \cite[Lemma~3.2]{CGG20}, allows us to write $z = \sum_{i=k}^{m_k} n_i q^i$ for some $n_k, \dots, n_{m_k} \in \ldb 0, a-1 \rdb$. Now take $r \in \ldb 0, s-t-1 \rdb$ such that $r_k = r$ for infinitely many $k \in \nn$. Then pick the sub-indices $k_1, \dots, k_{s-t}$ as follows. First, take $k_1 \ge N$ with $r_{k_1} = r$. Once $k_j$ has been chosen, take $k_{j+1} > m_{k_j}$ with $r_{k_{j+1}} = r$.
	
	Now consider the element $x = a^{k_1} + \dots + a^{k_{s-t}}$, and set $\ell_x := \min \mathsf{L}(x)$ and $L_x := \max \mathsf{L}(x)$. It is clear that $L_x = L_{k_1} + \dots + L_{k_{s-t}}$. %On the other hand, it follows from our choice of $k_1, \dots, k_{s-t}$ that in the (minimum-length) factorizations of $a^{k_1}, \dots, a^{k_{s-t}}$, the atoms do not intervene with each other. 
	We proceed to show that the identity $\ell_x = \ell_{k_1} + \cdots + \ell_{k_{s - t}}$ also holds. For $j\in \ldb 1, s - t \rdb$, let $\sum_{i = k_j}^{m_{k_j}} n_i^{(j)} q^i$ be the minimum-length factorization of $a^{k_j}$, where $n_i^{(j)}\in \ldb 0, a - 1 \rdb$ for every $i \in \ldb k_j, m_{k_j} \rdb$. Since $x = a^{k_1} + \dots + a^{k_{s-t}}$,
	\begin{equation} \label{eq:factorization of x}
		\sum_{j = 1}^{s - t}\sum_{i = k_j}^{m_{k_j}} n_i^{(j)} q^i \in \mathsf{Z}(x).
	\end{equation}
	Notice that from our construction we have
	\[
		\ldb k_1, m_{k_1} \rdb \prec \ldb k_2, m_{k_2} \rdb \prec \dots \prec \ldb k_{s-t}, m_{k_{s-t}} \rdb,
	\]
	where for any two subsets $S$ and $T$ of $\rr$, we write $S \prec T$ when $s < t$ for all $s \in S$ and $t \in T$. 
%	\[
%		k_1 < k_1 + 1 < k_1 + 2 < \cdots < m_{k_1} < k_2 < k_2 + 1 < \cdots < m_{k_2} < k_3 < \cdots < m_{k_{s - t - 1}} < k_{s - t} < k_{s - t} + 1 < \cdots < m_{k_{s - t}}.
%	\]
	Thus, in the factorization of $x$ show in~\eqref{eq:factorization of x}, for different ordered pairs $(i, j)$ and $(i', j')$, the coefficients $n_i^{(j)}$ and $n_{i'}^{(j')}$ correspond to different atoms, and all of the coefficients $n_i^{(j)}$ lay in the interval $\ldb 0, a - 1 \rdb$. It follows now from \cite[Lemma~3.2]{CGG20} that $x = \sum_{j = 1}^{s - t}\sum_{i = k_j}^{m_{k_j}} n_i^{(j)} q^i$ is the minimal factorization of $x$, hence,
	\[
		\ell_x = \sum_{j = 1}^{s - t}\sum_{i = k_j}^{m_{k_j}} n_i^{(j)} = \ell_{k_1} + \dots + \ell_{k_{s-t}},
	\]
	as desired. Therefore
	\begin{align*}
		tL_x - s \ell_x &= t \sum_{j=1}^{s-t} L_{k_j} - s \sum_{j=1}^{s-t} \ell_{k_j} = \sum_{j=1}^{s-t} (t L_{k_j} - s \ell_{k_j})\\
					   &\equiv \sum_{j=1}^{s-t} r_j \equiv (s-t)r \equiv 0 \! \! \! \pmod{s-t}.
	\end{align*}
	Since $\ell_x := \min \mathsf{L}(x)$ and $L_x := \max \mathsf{L}(x)$, the fact that $s-t \mid_\nn tL_x - s \ell_x$, along with the previous claim, guarantees the existence of $x' \in \nn_0[q]$ such that $\rho(x') = s/t$. Hence $\nn_0[q]$ is fully elastic.
%	Observe that $\frac{t+1}{t} \in S$ as after taking $k = b-1$ and $c = b(t(a-b)-1)$, we get
%	\[
%		\frac{t+1}{t} = \frac{a + (a-b)k + c}{b+c} \in R(\nn_0[q]).
%	\]
%	Now suppose that $s > t+1$. Take $n \in \nn$ large enough so that $t(a-b)n > b$, and then set $k = n(s-t) - 1$ and $c = t(a-b)n - b$. Note that both $k$ and $c$ are positive integers. In addition,
%	\begin{align*}
%		\frac{a + (a-b)k + c}{b+c} &= \frac{(a-b)n(s-t) + (a - (a-b) + c)}{b+c} \\
%													&= \frac{(a-b)n(s-t)}{t(a-b)n} + 1 = \frac st. 
%	\end{align*}
%	Therefore $s/t \in S \subseteq R(\nn_0[q])$.
\end{proof}

\bigskip
%%%%%%%%%%%%%%%
%%%%%%%%%%%%%%%
\section*{Acknowledgments}

We would like to thank our mentor, Felix Gotti, for suggesting this project and for his guidance all the way through. During the preparation of this paper, we were part of PRIMES-USA at MIT, and we are grateful to the organizers and directors of the program. Also, we would like to thank Kent Vashaw for reading our paper and providing helpful comments.

\bigskip
%%%%%%%%%%%%%%
%%%%%%%%%%%%%%

\end{document}